\documentclass[11pt,letterpaper]{amsart}

\usepackage{amsmath}
\usepackage{amsthm}
\usepackage{amssymb}
\usepackage{xfrac}
\usepackage{faktor}
\usepackage{mathtools}
\usepackage{mathrsfs}
\usepackage{tikz}
\usepackage{framed}
\usepackage{upgreek}
\usepackage{bbm}
\usepackage{bm}

\usepackage[pdftex,colorlinks,linkcolor=blue,citecolor=red,filecolor=black]{hyperref}

\usepackage{enumitem}
\setlist[enumerate]{label=({\roman*})}

\newtheorem{theorem}{Theorem}
\newtheorem{proposition}[theorem]{Proposition}
\newtheorem{corollary}[theorem]{Corollary}
\newtheorem{lemma}[theorem]{Lemma}

\theoremstyle{definition}

\newtheorem{remark}[theorem]{Remark}

\numberwithin{equation}{section}
\numberwithin{theorem}{section}

\def\<{\langle}
\def\>{\rangle}

\def\N{\mathbb{N}}

\def\Z{\mathbb{Z}}
\def\R{\mathbb{R}}

\def\C{\mathbb{C}}

\def\bGamma{\Bar{\Gamma}}

\begin{document}
\bibliographystyle{plain} \title[Periodic orbits in the homology group of a knot complement]
{Distribution of periodic orbits in the homology group of a knot complement} 

\author{Solly Coles} 
\address{Solly Coles, Department of Mathematics, Northwestern University,
Evanston, IL 60208, USA}
\email{solly.coles@northwestern.edu}

\author{Richard Sharp} 
\address{Richard Sharp, Mathematics Institute, University of Warwick,
Coventry CV4 7AL, UK}
\email{R.J.Sharp@warwick.ac.uk}




\begin{abstract}
Consider a transitive Anosov flow on a closed $3$-manifold. After removing a finite set of 
null-homologous periodic 
orbits, we study the distribution of the remaining periodic orbits in the homology 
of the knot complement. 
\end{abstract}

\maketitle

%
%
%
%
%
%
%
%
%

\section{Introduction}

In a recent paper \cite{mcmullen}, McMullen discussed the way in which periodic orbits 
$\mathscr P$ of the geodesic
flow on the unit-tangent bundle of a compact negatively curved surface (and more general Anosov flows on 3-manifolds) behave like prime numbers, in the sense that they obey a 
\emph{Chebotarev law}. More precisely, if $M_N$ is the knot complement obtained by removing
$N$ periodic orbits $\{K_1,\ldots,K_N\}$ and $G$ is a finite quotient of $\pi_1(M_N)$ then
each of the remaining periodic orbits $\gamma \in \mathscr P_N :=\mathscr P\setminus \{K_1,\ldots,K_N\}$ induces a well-defined conjugacy class $\langle \gamma\rangle_G$
in $G$ and these
conjugacy classes are equidistributed in the sense that, writing $\ell(\gamma)$ for the least period of $\gamma$, for every conjugacy class $C \subset G$ we have
\[
\lim_{T \to \infty} \frac{|\{\gamma \in \mathscr P_N \hbox{ : } \ell(\gamma) \le T \,
\langle \gamma \rangle_G = C\}|}{|\{\gamma \in \mathscr P_N \hbox{ : } \ell(\gamma) \le T\}|}
= \frac{|C|}{|G|},
\]
where $|A|$ denotes the cardinality of $A$.
(See also Ueki \cite{ueki} for results on modular knots.)
The aim of this paper is to provide precise asymptotic information about how the periodic orbits in $\mathscr P_N$ link with the removed orbits $\{K_1,\ldots,K_N\}$.
This may be viewed as a question about how the orbits are distributed in the homology of
the knot complement $H_1(M_N,\Z)$.
The distribution of periodic orbits in the first homology group of $M$ is well understood 
(\cite{Anan}, \cite{bab-led}, \cite{KS}, \cite{Kotani}, \cite{lalley}, \cite{Sh93}) and the
above problem
may be viewed as a generalisation of this theory.
Many of the ideas introduced by McMullen will be invaluable to our analysis.

We shall begin by recalling some classical distribution results for periodic orbits.
Let $M$ be a smooth closed (compact and without boundary) oriented Riemannian 3-manifold
and let $X^t : M \to M$ be a weak-mixing transitive Anosov flow. This flow has a countably infinite set 
$\mathscr P$ of prime periodic orbits: we denote a typical such orbits by $\gamma$ and its least 
period by $\ell(\gamma)$. The following asymptotic formula is due to Margulis \cite{margulis1}, \cite{margulis2}:
\[
|\{\gamma \in \mathscr P \hbox{ : } \ell(\gamma) \le T\}| \sim \frac{e^{h(X)T}}{h(X)T},
\quad \mbox{as } T \to \infty,
\]
where $h(X)$ is the topological entropy of $X$. 

We may also consider how the periodic orbits are distributed with respect to the homology of $M$.
Let $b=b_1(M)$ denote the first Betti number of $M$,
which we shall assume to be at least $1$.
For $\gamma \in \mathcal P$, we shall write $[\gamma]$ for the torsion-free part of the homology class
of $\gamma$ in $H_1(M,\Z)$; we may identify this with an element of $\Z^b$. Write
$\mathscr C$ for the closure of the set of normalised homology classes,
\begin{equation}\label{def of scrC}
\mathscr{C} = \overline{\left\{\frac{[\gamma]}{\ell(\gamma)} \hbox{ : } \gamma \in \mathscr P\right\}} \subset \R^b.
\end{equation}
This is a compact convex set with non-empty interior. We say that $X^t$ is 
{\it homologically full} if the interior of $\mathscr C$ contains zero.
Assuming for the moment that $X$ is homologically full,
write $\mathscr P_0 = \{\gamma \in \mathscr P \hbox{ : } [\gamma] =0\}$.
Then there exists $0<h^* \le h(X)$ and $C>0$ such that
\[
|\{\gamma \in \mathscr P_0 \hbox{ : } \ell(\gamma) \le T\}| \sim C \frac{e^{h^*T}}{T^{1+b/2}},
\quad \mbox{as } T \to \infty,
\]
and a similar formula for $[\gamma]=\alpha$ for any fixed $\alpha \in H_1(M,\Z)$
(though the constant $C$ may change)
\cite{Sh93}. 

Dropping the assumption that $X$ is homologically full, more general results hold where, instead of restricting $\gamma$ to be null-homologous (or to lie in a fixed homology class), 
we consider $\gamma$ such that $[\gamma]$ is approximately $T\rho$, where $\rho \in \mathrm{int}(\mathscr C)$ (Lalley \cite{lalley} for geodesic flows over surfaces, Babillot and Ledrappier \cite{bab-led} for the general case). The right hand side then takes the form
$
C_\rho(T) e^{H(\rho)T}T^{-(1+b/2)},
$
where $H : \mathrm{int}(\mathscr C) \to \mathbb R^+$ is a natural entropy function and $C_\rho(T)$
an explicit continuous function which 
is bounded above and bounded below away from zero.

The object of this paper will be to consider analogous problems for the knot complements
introduced above, inspired by the approach of McMullen
\cite{mcmullen}.
Let $K_1,\ldots,K_n \in \mathscr P_0$ be a finite set of (integral) null-homologous periodic orbits,
which we think of as knots in $M$, and consider
the knot complement $M_N := M\setminus \{K_1,\ldots,K_N\}$. Then, 
for any $\gamma \in \mathscr P_N := \mathscr P\setminus \{K_1,\ldots,K_N\}$, 
we can consider the torsion-free part of its homology class
$[\gamma]_N \in H_1(M_N,\Z)/\mathrm{torsion}$.
Let 
\[
\mathscr C_N = \overline{\left\{\frac{[\gamma]_N}{\ell(\gamma)} \hbox{ : } \gamma \in \mathscr P_N\right\}} \subset H_1(M_N,\R) \cong \R^{b+N}.
\]
Write
\[
\pi_N(T,\lfloor \rho T\rfloor) := 
|\{\gamma \in \mathscr P_N \hbox{ : } \ell(\gamma) \le T, \ 
[\gamma]_N= \lfloor \rho T \rfloor
\}|,
\]
where $\lfloor \cdot \rfloor$ is defined
by choosing a fundamental domain $\mathfrak F_N$ for $H_1(M_N,\Z)$ as a lattice
in $H_1(M_N,\R)$ and setting $\lfloor \rho \rfloor$ to be the unique
element of $H_1(M_N,\Z)/\mathrm{torsion}$ satisfying $\rho - \lfloor \rho \rfloor \in \mathfrak F_N$.
Here is a sample theorem.

\begin{theorem}\label{thm:main-intro}
Let $M$ be a smooth closed oriented Riemannian $3$-manifold and let $X^t : M \to M$ be 
a transitive Anosov flow which mixes exponentially fast. There is a strictly positive
real analytic function $\mathfrak h_N : \mathrm{int}(\mathscr C_N) \to \R^+$ such that, for each
$\rho \in \mathrm{int}(\mathscr C_N)$, we have
\[
\pi_N(T,\lfloor \rho T\rfloor)
\sim c_\rho(T) \frac{e^{\mathfrak h_N(\rho)T}}{T^{1+(b+N)/2}}, 
\]
as $T \to \infty$,
where $c_\rho(T)$ is bounded above and bounded below away from zero.
\end{theorem}

\begin{remark}
In fact, the result holds under a weaker hypothesis than exponential mixing.
See section \ref{sec:anosovflows} below.
\end{remark}

We will now outline the content of the rest of the paper. In section \ref{sec:anosovflows}, we define Anosov flows and introduce the assumptions for our main result. In section \ref{sec:knotcomp}, we discuss the homology of a link complement, when the link is given by finitely many null-homologous periodic orbits in $M$. In particular we show that the first homology group is generated by the homology classes of the remaining periodic orbits. In section \ref{sec:symbolic_dynamics}, we see how the homology of orbits is encoded by the modelling of our Anosov flow as a symbolic suspension flow. In section \ref{therm_form}, we describe some relevant tools from the thermodynamic formalism for suspension flows. In section \ref{sec:counting} we state our main results and show how they follow from the main result of \cite{bab-led}.

\section{Anosov flows}\label{sec:anosovflows}
Let $M$ be a smooth, connected, oriented, closed (compact and without boundary) Riemannian 3-manifold  and let $X^t:M\to M$ be a $C^1$ flow generated by the vector field $X$.
We will assume $X^t$ is Anosov, meaning that the tangent bundle has a continuous $DX^t$-invariant splitting $TM = E^0 \oplus E^s \oplus E^u$,
where $E^0$ is the one-dimensional bundle spanned by $X$ and where there exist constants $C,\lambda >0$ such that
\begin{enumerate}
\item
$\|DX^tv\| \le Ce^{-\lambda t} \|v\|$, for all $v \in E^s$ and $t >0$;
\item
$\|DX^{-t}v\| \le Ce^{-\lambda t} \|v\|$, for all $v \in E^u$ and $t>0$.
\end{enumerate}
This class of flows was introduced by Anosov \cite{anosov}; for a good modern reference
see \cite{fisher}.
In addition, we assume that $X^t : M \to M$ is topologically transitive, i.e. that there is a dense orbit. 

We say that the flow is {\it topologically weak-mixing} if the equation 
$\psi \circ X^t = e^{iat} \psi$, for $\psi\in C(M,\C)$ and $a \in \mathbb R$, has only the trivial solution where $\psi$ is constant and $a=0$. We say it is {\it topologically mixing} if for all
non-empty open sets $U,V \subset M$, there exists $t_0$ such that
$X^t(U) \cap V \ne \varnothing$ for all $t \ge t_0$.
For Anosov flows (and more general hyperbolic flows) these conditions are equivalent
and are also equivalent to the flow being mixing with respect to the equilibrium state $\mu$ of any
H\"older continuous function, i.e. that
\[
\rho_{F,G}(t) := 
\int F \circ X^t \, G \, d\mu - \int F \, d\mu \int G \, d\mu \to 0,
\]
as $t \to \infty$, for all $F,G \in L^2(\mu)$.
A classical result of Plante \cite{plante} is that a transitive Anosov flow fails to be topologically weak-mixing if and 
only if it is the constant time suspension of an Anosov diffeomorphism.

An Anosov flow has a countably infinite set $\mathscr P$ of prime periodic orbits. 
For $\gamma \in \mathscr P$, $\ell(\gamma)$ denotes the least period of $\gamma$.
We have that the flow fails to be mixing if and only if
$\{\ell(\gamma) \hbox{ : } \gamma \in \mathcal P\}$ is contained in a discrete subgroup of $\mathbb R$.
If $X^t : M \to M$ is mixing then $|\{\gamma \in \mathscr P \hbox{ : } \ell(\gamma) \le T\}|
\sim e^{hT}/hT$, where $h$ is the topological entropy \cite{margulis1},\cite{margulis2},\cite{PP83}.

A transitive Anosov flow is 
(bounded-to-one) semiconjugate to a suspension flow over a subshift of finite type.
(This will be discussed in greater detail in section \ref{sec:symbolic_dynamics}.)
We shall consider flows such that the roof function cannot be chosen to be locally constant;
we call such flows {\it not locally constant}.
This is an open dense condition which is implied by the flow having {\it good asymptotics} in the sense defined in \cite{FMT}.
(We omit the rather complicated definition.)
In particular, this holds if the flow is exponentially mixing, i.e. that $\rho_{F,G}(t)$ converges to zero exponentially fast for the equilibrium state of any H\"older continuous function and all
sufficiently regular $F$ and $G$.
The class of exponentially mixing transitive Anosov flows includes
geodesic flows over compact surfaces of (not necessarily constant)
negative curvature \cite{Dol1}. 
In fact, Tsujii and Zhang have shown that any {\it smooth} mixing Anosov flow on a closed $3$-manifold is
exponentially mixing \cite{Tsujii-Zhang}.

\section{Knot complements}\label{sec:knotcomp}

We now wish to consider $M$ with a finite number of null-homologous periodic orbits removed. 
(For convenience, we suppose that they are integrally null-homologous, so
we end up with integral linking numbers.)
Let $K_1,\ldots,K_N$ be null-homologous knots in $M$ (for the moment, they do not need to be 
periodic orbits).
We need to understand the homology of the complement $M_N=M\setminus \bigcup_{i=1}^N K_i$.
We will denote the homology class of $\gamma$ in $H_1(M_N,\R)$ by
$[\gamma]_N$. 

 For each $i=1,\ldots,N$, replace $K_i$ with
a tubular neighbourhood
$\mathscr N(K_i)$ (an embedded solid torus), and (slightly abusing notation) let
\[
M_N := M \setminus \bigcup_{i=1}^N \mathrm{int}(\mathscr N(K_i)).
\]
For each $i=1,\ldots,N$, let $\mathfrak m_i$ be a meridian in $\partial \mathscr N(K_i)$.

\begin{lemma}
Let $K_1,\ldots,K_N$ be null-homologous knots in $M$. Then
\[
H_1(M_N,\R) \cong H_1(M,\R) \oplus \R^N.
\]
In particular, the first Betti number of $M_N$ is $b+N$.
Furthermore, the classes $[\mathfrak m_1]_N,\ldots,[\mathfrak m_N]_N$ are linearly independent in
$H_1(M_N,\R)$.
\end{lemma}

\begin{proof} The first statement follows from the 
homology long exact sequence
\begin{align*}
 \longrightarrow H_2(M_N,\partial M_N) \longrightarrow H_1(\partial M_N)
\longrightarrow H_1(M_N) \longrightarrow H_1(M_N,\partial M_N) \longrightarrow 
\end{align*}
and the `half lives, half dies' principle (Lemma 3.5 of \cite{hatcher-notes}),
which tells us that the dimension of the image 
of the boundary homomorphism
$\partial: H_2(M_N,\partial M_N) \to H_1(\partial M_N)$ is equal to $N$.

For the second statement, we are indebted to a Mathematics Stack Exchange post by Kyle Miller.
By construction, the longitude $\mathfrak l_i$ is null-homologous in 
$M\setminus \mathrm{int}(\mathscr N(K_i))$ and so we can choose an embedded
surface $S_i \subset M\setminus \mathrm{int}(\mathscr N(K_i))$ with boundary $\mathfrak l_i$.
We may assume that $S_i$ intersects the other knots transversally and so we obtain an embedded
surface $S_i' = S_i \cap M_N$ that represents $[\mathfrak l_i]$ as a linear combination of the
homology classes of the meridians
$[\mathfrak m_1]_N,\ldots,[\mathfrak m_N]_N$.
Since $[\mathfrak m_1]_N,\ldots,[\mathfrak m_N]_N,[\mathfrak l_1]_N,\ldots,[\mathfrak l_N]_N$
is a basis for $H_1(\partial M_N,\R)$, $[\mathfrak m_1]_N,\ldots,[\mathfrak m_N]_N$ span the image
of the inclusion homomorphism $i_* : H_1(\partial M_N,\R) \to H_1(M_N,\R)$, and
$\dim \ker i_* = N$, it follows that  
$[\mathfrak m_1]_N,\ldots,[\mathfrak m_N]_N$ are linearly independent in $H_1(M_N,\R)$.
\end{proof}

We will need the following result. (The analogous statement for $H_1(N,\Z)$ may be found in \cite{PP86}.)

\begin{proposition}\label{prop:periodic_orbits_generate_homology}
The torsion-free part of the group $H_1(M_N,\mathbb Z)$ is generated
by $\{[\gamma]_N\hbox{ : }\gamma\in \mathscr P_N\}$.
\end{proposition}

\begin{proof}
This is a corollary of the main result in \cite{mcmullen}, where an analogue of the Chebotarev density theorem is proved for weak-mixing Anosov flows. 
More precisely, for a finite group $G$, and a surjective homomorphism $\phi :\pi_1(M_N)\to G$, given any conjugacy class $C$ in $G$ we have
\begin{equation}\label{chebotarev}
\lim_{T\to\infty}\frac{|\{\gamma\in \mathscr P_N\hbox{ : }\ell(\gamma)\leq T,\, \phi(\gamma)\in C\}|}
{|\{\gamma\in \mathscr P_N\hbox{ : }\ell(\gamma)\leq T\}|}=\frac{|C|}{|G|}.
\end{equation}
If $H_1(M_N,\mathbb Z)/\mathrm{torsion}$ is not generated by the orbits in $\mathscr P_N$ 
then there is some proper cofinite subgroup 
$H\leq H_1(M_N,\mathbb Z),$ such that 
$$\{[\gamma]_N\hbox{ : }\gamma\in \mathscr P_N\}\subset H.
$$ 
If we set $G=H_1(M_N,\mathbb Z)/H$ then we have a quotient homomorphism
$\phi : \pi_1(M_N) \to G$
and the formula (\ref{chebotarev}) is contradicted, since all orbits 
in $\mathscr P_N$ satisfy $\phi(\gamma)=0$.
\end{proof}

\section{Symbolic dynamics}\label{sec:symbolic_dynamics}
In this section, we discuss how to model our flows by symbolic systems, which 
(as we will see in the next section) will also keep track of the homology and linking numbers of periodic orbits.
This is classical material, due independently to Bowen and Ratner. 
We begin by defining subshifts of finite type.

\subsection{Subshifts of finite type and coding}
Let $k\geq 2$, and let $\Gamma$ be a directed graph on $k$ vertices
$V(\Gamma)$, labelled $1,\ldots,k$, and with directed edges $E(\Gamma) \subset V(\Gamma) \times V(\Gamma)$. 
(In particular, for each ordered pair of vertices $(i,j)$ there is at most one directed edge from $i$ to $j$.)
Define $$\Sigma(\Gamma):=\{(x_n)_{n\in\Z}\in V(\Gamma)^{\Z} : (x_n,x_{n+1})\in E(\Gamma)\text{ for all }n\},$$ i.e. the set of bi-infinite walks on $\Gamma.$
The shift map 
$\sigma:\Sigma(\Gamma)\to\Sigma(\Gamma)$ is defined by 
$\sigma((x_{n})_{n\in \Z})=(x_{n+1})_{n\in \Z}$. 

We give $V(\Gamma)$ the discrete topology, $V(\Gamma)^{\mathbb Z}$ the product topology, and 
$\Sigma(\Gamma)$ the subspace topology.
The following metrics are compatible with this topology.
Fix $0<\theta <1$. For $x,y\in X$ set $d_\theta(x,y)=\theta^{n(x,y)}$, where 
\[
n(x,y) =\sup\{n \in \mathbb Z^+ \hbox{ : } x_i=y_i \ \forall |i|<n\}.
\]
With this metric in place, let $F_\theta$ denote the set of real-valued $d_\theta$-Lipschitz functions on 
$\Sigma(\Gamma)$. 

We will say that $\Gamma$ is {\it aperiodic} if there exists $N \ge 1$ such that, for each ordered
pair of vertices $(i,j)$, there is a path of length $N$ from $i$ to $j$. This condition is equivalent 
to $\sigma : \Sigma(\Gamma) \to \Sigma(\Gamma)$ being topologically mixing (i.e. that there
exists $N \ge 1$ such that for all non-empty open sets $U,V \subset \Sigma(\Gamma)$, we have
$\sigma^n(U) \cap V \ne \varnothing$ for all $n \ge N$).

Given a strictly positive function $r\in F_\theta$, we can suspend 
$\Sigma(\Gamma)$ with $r$ as roof function, to obtain the space 
$$\Sigma(\Gamma,r)=
\{(x,t)\in \Sigma(\Gamma)\times \R\hbox{ : }0\leq t\leq r(x)\}/
(x,r(x))\sim (\sigma(x),0)
.$$ 
The suspension flow $\sigma^r_t:\Sigma(\Gamma,r)\rightarrow \Sigma(\Gamma,r)$ is defined by $\sigma^r_t[x,s]=[x,s+t],$ where $[x,s]$ is the equivalence class of $(x,s)$ in $\Sigma(\Gamma,r)$. 
According to the classic results of Bowen and Ratner from the 1970s, Anosovs flows (and more general hyperbolic flows) may be modelled by suspension flows of this type.

\begin{theorem}[Bowen \cite{bowen2}, Ratner \cite{ratner}]\label{symbolic} Let $X^t:M\rightarrow M$ be a topologically weak-mixing transitive Anosov flow. Then there exists a suspension flow $(\sigma^r,\Sigma(\Gamma,r))$,
over a topologically mixing  $(\sigma,\Sigma(\Gamma))$,
and a continuous surjection $\pi:~\Sigma(\Gamma,r)\rightarrow M$ such that
\begin{enumerate}
    \item[\emph{(i)}] For all $t\in \R$, $\pi\circ\sigma^r_t=X^t\circ\pi$;
    \item[\emph{(ii)}] There exists $N\in \N$ such that $\#\pi^{-1}(y)\leq N$ for all $y\in M$;
    \item[\emph{(iii)}] $\sigma^r$ is topologically weak-mixing and transitive.
\end{enumerate}
\end{theorem}

The map $\pi$ is constructed as follows. Using the structure of the stable and unstable manifolds of $X$, one can construct a Markov section. In particular, this gives a finite set
$\{R_1,\ldots R_k\}$ of codimension one cross-sections to the flow (or rectangles), such that every orbit of $X^t$ intersects $R=\bigcup_{i=1}^kR_i$ infinitely often in the past and future, and 
these intersections are transversal. 
Letting $\tau:R\to R$ be the first return map to $R$ (or Poincar\'e map), define a directed graph $\Gamma$ on $k$ vertices $\{1,2,\ldots, k\}$ as follows. First, denote by $ij$ the directed edge from vertex $i$ to vertex $j$. Include $ij$ in the directed edge set $E(\Gamma)$ whenever $R_i\cap \tau^{-1}(R_j)\neq \varnothing$. This yields a connected directed graph, with at most two edges (of opposite direction) between two vertices. 
The rectangles are chosen with the properties that there exists $T>0$ such that $M\subset X^{[0,T]}R$, and that , for each $x\in \Sigma(\Gamma)$, $\bigcap_{n\in\Z}\tau^n(R_{x_n})$ contains exactly one point. With this, we can define $\pi':\Sigma(\Gamma)\to R$ by $\pi'(x)\in \bigcap_{n\in\Z}\tau^n(R_{x_n}).$ Setting $r:R\to \R^+$ to be the first return time to $R$, we can extend $\pi'$ to a map $\pi:\Sigma(\Gamma,r)\to M$, by $\pi[x,s]=X^s(\pi'(x)).$

\subsection{Homology of periodic orbits via symbolic dynamics}

We need to understand how periodic orbits for $X^t$
are encoded by $\pi$ and how homology and linking may be captured by the symbolic model.

Given $1\leq i,j\leq k$, let $E_{ij}$ be the set of points on flow lines going from 
$\text{int}(R_i)$ to $\text{int}(R_j),$ and let 
$$
U=\left(\bigcup_{i=1}^k\text{int}(R_i)\right)\cup\left(\bigcup_{1\leq i,j\leq k}E_{ij}\right).
$$

Let $\mathscr P_{\sigma^r}$ denote the set of prime periodic orbits for $\sigma^r$.
Given a periodic orbit $\gamma \in \mathscr P$, there is a periodic orbit 
$\eta \in \mathscr P_{\sigma^r}$, with $\ell(\gamma)=\ell(\eta)$, such that $\gamma=\pi(\eta)$. Furthermore, $\eta$ is unique as long as $\gamma\subset U$, i.e. $\gamma$ does not pass through the edge of some rectangle. In the particular case of a 3-manifold, uniqueness holds except for a finite number of orbits.
(In higher dimensions, there many be infinitely many non-uniquely coded orbits 
but they exhibit a slower rate of exponential growth.)
Thus $\pi$ induces a one-to-one correspondence for periodic orbits in $\mathscr P$ and $\mathscr P_{\sigma^r}$, up to removing finitely many.

An additional feature of the symbolic dynamics is that,
as argued in \cite{mcmullen}, we can arrange for 
the removed periodic orbits $K_1,\ldots,K_N$  to sit in the boundary of the sections.
Indeed, suppose that, for some $1\le n \le N$, $x\in K_n \cap \hbox{int}(R_i)$ (by transversality, there can only be finitely many such $x$). 
Then we may split $R_i$ into two pieces $R_{i,1}$ and $R_{i,2}$ at the point $x$,
so that
\[
\{R_1,\ldots, R_{i-1},R_{i,1},R_{i,2},R_{i+1},\ldots,R_k\}
\] 
still forms a Markov section with the same properties as above. With this in mind, we will henceforth assume that $K_1,\ldots,K_N$ are in the boundary of $R$.

Let $\Bar{\Gamma}$ be the graph obtained by removing the direction from edges of $\Gamma$, but retaining any multiple edges between vertices. 
(So if $ij$ and $ji$ are both directed edges then $\Bar \Gamma$ has two undirected edges between $i$ and $j$.)
One can then choose a natural embedding $\iota:\Bar{\Gamma}\hookrightarrow U\subset M_N$. 
The following result appears as Lemma 5.1 of \cite{mcmullen}.

\begin{proposition}\label{prop:surjection} The embedding $\iota:\bGamma\to U$ induces a surjective homomorphism $\iota_*:\pi_1(\Bar{\Gamma})\to\pi_1(M_N).$ \end{proposition}

We now describe how the homology of periodic orbits can be encoded using the symbolic dynamics.
For each $i\in\{1,\ldots, k\}$, fix a path $p(1,i)$ from $1$ to $i$ in $\Bar{\Gamma}$, 
noting that paths in $\Bar{\Gamma}$ need not follow the directions of edges of $\Gamma$. 
That such paths exist for all $i$ is a consequence of transitivity of $\sigma.$ Let $p(i,1)$ be the path from $i$ to $1$ obtained by following $p(1,i)$ backwards. For any vertex $j$ which satisfies that $ij\in E(\Gamma)$, let $e(i,j)$ be the corresponding edge between $i,j$ in $\bGamma$. 
Then we can form a loop $K(i,j)$
in $\bGamma$ given by 
$$
K(i,j)=1\xrightarrow{p(1,i)}i\xrightarrow{e(i,j)}j\xrightarrow{p(j,1)}1,
$$ 
which has a homotopy class $\langle K(i,j) \rangle \in \pi_1(\bGamma).$  
From Proposition \ref{prop:surjection}, we have a surjective homomorphism
$
\iota_*:\pi_1(\bGamma)\to\pi_1(M_N).
$
We also have the Hurewicz homomorphism
$$
q:\pi_1(M_N)\to H_1(M_N,\Z)
$$ 
and the projection homomorphism 
$$
\mathrm{proj}:H_1(M_N,\Z)\to H_1(M_N,\Z)/\mathrm{torsion} \cong \Z^{b+N}.
$$ 
Let $\rho=\mathrm{proj}\circ q\circ \iota_*:\pi_1(\bGamma)\to \Z^{b+N},$ 
and define $f:\Sigma(\Gamma)\to\Z^{b+N}$ by 
\begin{equation}\label{hom-function}
f(x)=\rho([K(x_0,x_1)]).
\end{equation}
Clearly, $f$ is locally constant. Further, as $\rho$ is a homomorphism, 
we have that for a periodic point $x=\overline{x_0x_1\ldots x_{n-1}},$ the Birkhoff sum 
\begin{align*}
f^n(x) :&=f(x)+f(\sigma x)+\ldots +f(\sigma^{n-1}x)\\
&=\rho([K(x_0,x_1)]_N [K(x_1,x_2)]_N \cdots[K(x_{n-2},x_{n-1})]_N [K(x_{n-1},x_0)]_N)\\
&=\rho([c(x_0,x_1,\ldots, x_{n-1},x_0)]_N), 
\end{align*}
where $c(x_0,x_1,\ldots, x_{n-1},x_0)$ is the cycle $$x_0\xrightarrow{e(x_0,x_1)}x_1\xrightarrow{e(x_1,x_2)}x_2\to\ldots\to x_{n-2}\xrightarrow{e(x_{n-2},x_{n-1})}x_{n-1}\xrightarrow{e(x_{n-1},x_0)}x_0.$$  This construction leads to the following.

\begin{lemma}
Let $\eta\in \mathscr P_{\sigma^r},$ and $x\in \Sigma(\Gamma)$ the corresponding periodic point for $\sigma$, with period $n\in \N.$ Then $\ell(\pi(\eta))=\ell(\eta)=r^n(x),$ and $f^n(x)=[\pi(\eta)]_N$.
\end{lemma}

\section{Thermodynamic formalism for suspension flows}\label{therm_form}

To formulate and prove our results, we will need to use the machinery of thermodynamic formalism.
Since we need to deal with linking numbers and homology of the knot complement $M_N$, we will not be able to use thermodynamic functions defined with respect to the Anosov flow $X^t :M \to M$. Rather, we will work with the shift $\sigma : \Sigma(\Gamma) \to \Sigma(\Gamma)$ and suspension
flow $\sigma_t^r : \Sigma(\Gamma,r) \to \Sigma(\Gamma,r)$.

Let $\mathscr M_\sigma$ denote the space of $\sigma$-invariant Borel probability
measures on $\Sigma(\Gamma)$.
For a continuous function $\varphi : \Sigma(\Gamma) \to \R$, we define its
{\it pressure} (with respect to $\sigma$) $P_\sigma(\varphi)$ by
\[
P_\sigma(\varphi) = \sup_{m \in \mathscr M_\sigma} \left\{h_m(\sigma) + \int \varphi \, dm\right\},
\]
where $h_m(\sigma)$ is the measure theoretic entropy. If $\varphi$ is H\"older continuous then 
the supremum is attained at a unique measure $m_\varphi$, which is called the equilibrium state for
$\varphi$; this measure is ergodic and fully supported.

Now let $\mathscr M_{\sigma^r}$ denote the space of $\sigma^r$-invariant Borel probability
measures on $\Sigma(\Gamma,r)$. There is a natural bijection between $\mathscr M_{\sigma}$
and $\mathscr M_{\sigma^r}$ given by $m \mapsto \widetilde m$, where $\widetilde m$ is given locally by
\[
\frac{m \times \mathrm{Leb}}{\int r \, dm}.
\]
For a continuous function $\Phi : \Sigma(\Gamma,r) \to \R$, we define its pressure (with respect
to $\sigma^r$) $P_{\sigma^r}(\Phi)$ by
\begin{equation}\label{def:pressure_suspension_flow}
P_{\sigma^r}(\Phi) = \sup_{\mu \in \mathscr M_{\sigma^r}} \left\{h_\mu(\sigma^r) + \int \Phi \, d\mu\right\}.
\end{equation}
One can put a metric on $\Sigma(\Gamma,r)$ 
\cite{bowen-walters}
and consider H\"older continuous functions with respect to this metric, but for our purposes it is 
sufficient to consider functions satisfying the following (strictly weaker) condition. We say that
$\Phi$ is {\it
fibre-H\"older} if $\varphi : \Sigma(\Gamma) \to \R$, defined by
\begin{equation}\label{fibre-integral}
\varphi(x) = \int_0^{r(x)} \Phi(x,t) \, dt,
\end{equation}
is H\"older continuous. If $\Phi$ is fibre-H\"older then the supremum in 
(\ref{def:pressure_suspension_flow}) is attained at a unique measure $\mu_\Phi$, which is called the equilibrium state for
$\Phi$; this measure is ergodic and fully supported.
Furthermore, $\mu_\Phi$ is given locally by the product
\[
\frac{m_{-P(\Phi)r+\varphi} \times \mathrm{Leb}}{\int r \, dm_{-P(\Phi)r+\varphi}},
\]
where $\varphi$ is the function in (\ref{fibre-integral}).

It will also be convenient to pass from functions on
$\Sigma(\Gamma)$ to functions on $\Sigma(\Gamma,r)$, so we `lift' 
functions in the following way. 
Choose a smooth function $\kappa : [0,1] \to \R^+$ satisfying
$\kappa(0)=\kappa(1) =0$ and $\int_0^1 \kappa(t) \, dt=1$.
Let $\varphi:\Sigma(\Gamma)\to \R$ be H\"older
continuous and define $\widetilde{\varphi}:\Sigma(\Gamma,r)\to \R$ by 
$$
\widetilde{\varphi}(x,t)=\frac{\varphi(x)}{r(x)} \kappa\left(\frac{t}{r(x)}\right).
$$
Then
\[
\varphi(x)=\int_{0}^{r(x)}\widetilde{\varphi}(x,t)\, dt
\]
and $\widetilde \varphi$ is automatically fibre-H\"older. 
Furthermore, for a periodic orbit $\eta$ of $\sigma^r$, corresponding to a point $x\in \Sigma(\Gamma)$ of least period $n$, 
$$\frac{\varphi^n(x)}{r^n(x)}=\frac{1}{\ell(\eta)}\int_\eta\widetilde{\varphi},$$
where 
$\int_\eta\tilde{\varphi}=\int_0^{\ell(\eta)}\tilde{\varphi}(\sigma^r_tx_\eta)\,dt$, for $x_\eta\in \eta.$ This integral defines a probability measure, $\mu_\eta$, by 
$$\int \Phi \, d\mu_\eta =\frac{1}{\ell(\eta)}\int_\eta\Phi.$$

We now wish to consider pressure functions of a particular form. Fix a fibre-H\"older function
$F : \Sigma(\Gamma,r) \to \R^d$ (for some $d \ge 1$) and let $f : \Sigma(\Gamma) \to \R^d$ be defined by the higher dimensional analogue of (\ref{fibre-integral}), i.e.
\[
f(x) = \int_0^{r(x)} F(x,t) \, dt.
\]
We then have a set $\mathscr C_F \subset \R^d$ defined in the following equivalent ways:
\begin{align*}
\mathscr C_F &= \overline{\left\{\frac{1}{\ell(\eta)} \int_\eta  F \hbox{ : } \eta \in \mathscr P_{\sigma^r}
\right\}}
= \left\{\int F d\mu \hbox{ : } \mu \in \mathscr M_{\sigma^r} \right\}
\\
&= \left\{
\frac{\int f \, dm}{\int r \, dm} \hbox{ : } m \in \mathscr M_\sigma\right\}.
\end{align*}

We write $F = (F_1,\ldots,F_d)$ and, for $u = (u_1,\ldots,u_d) \in \R^d$, write
$\langle u,F \rangle = u_1F_1 + \cdots + u_dF_d$.
We then define the pressure function
$\mathfrak p_F : \R^d \to \R$  by
\[
\mathfrak p_F(u) = P_{\sigma^r}(\langle u,F \rangle).
\]
This is real analytic and convex and satisfies
\[
\nabla \mathfrak p_F(u) = \int F \, d\mu_{\langle u,F \rangle}.
\]
We have the following standard result (see \cite{lalley87}, for example).

\begin{lemma}\label{lem:strictly_convex}
The following are equivalent.
\begin{enumerate}
\item
For all $u \in \R^d\setminus \{0\}$, the function $\langle u,F\rangle : \Sigma(\Gamma,r) \to \R$ is 
not cohomologous to a constant.
\item
The pressure function $\mathfrak p_F$ is strictly convex.
\item
$\mathrm{int}(\mathscr C_F) \ne \varnothing$.
\end{enumerate}
\end{lemma}

Assuming that any (and hence all) of the statements in Lemma \ref{lem:strictly_convex} hold, we define
an entropy function
$\mathfrak h_F : \mathrm{int}(\mathscr C_F) \to \mathbb R$ by
\[
\mathfrak h_F(\rho)=\sup\bigg\{h_{\mu}(\sigma)\hbox{ : }\mu \in \mathscr M_{\sigma^r},\,\int F\,d\mu=\rho \bigg\}.
\] 
This leads to the following lemma.

\begin{lemma}\label{intofC}
The map $u\mapsto \nabla \mathfrak p_F(u)$ is a smooth diffeomorphism between $\R^{d}$ and 
$\mathrm{int}(\mathscr C_F)$. Furthermore, $\mathfrak h_F$ is a smooth function on 
$\mathrm{int}(\mathscr C_F)$ and $(\nabla \mathfrak p_F)^{-1}=-\nabla \mathfrak h_F$.
\end{lemma}

\begin{proof}
This essentially follows from Theorem 26.5 of \cite{Rock}. Since
\[
\mathfrak p_F(u) = h_{\mu_{\langle u,F\rangle}}(\sigma^r) + \left\langle u,\int F \, d\mu_{\langle u,F \rangle} \right\rangle
\]
for a unique measure $\mu_{\langle u,F \rangle} \in \mathscr{M}_{\sigma^r}$,
we have
\[
\mathfrak p_F(u) = \sup_{z \in \mathscr C_F} \Big(\mathfrak h_F(z) + \langle u,z \rangle \Big),
\]
with the supremum attained uniquely at $z = \int F \, d\mu_{\langle u,F \rangle}$.
This shows that the functions $\mathfrak p_F$ and $-\mathfrak h_F$ are Legendre duals.
Theorem 26.5 of \cite{Rock} then gives that 
$\nabla \mathfrak p_F$ is a diffeomorphism onto its image with inverse
$-\nabla \mathfrak h_F$. To complete the proof, it only remains to show that
$\nabla \mathfrak p_F(\mathbb{R}^d) = \mathrm{int}(\mathscr C_F)$.

We now use an argument adapted from the proof of Lemma 7 in \cite{marcus-tuncel}.
Let $z \in \mathrm{int}(\mathscr C_F)$ and choose $\epsilon >0$ such that
$B(z,2\epsilon) \subset \mathscr C_F$.
Then, for $u \ne 0$, 
$z+ (u/\|u\|)\epsilon \in \mathscr C_F$.
Thus
\[
\langle u,z \rangle + \|u\| \epsilon \le
\sup_{w \in \mathscr C_F} \langle u,w \rangle
= \sup\left\{\left\langle u,\int F \, d\mu \right \rangle \hbox{ : } \mu \in \mathscr M(\sigma^r)\right\}
\le \mathfrak p_F(u).
\]
Let $e_z : \R^d \to \R$ be defined by $e_z(u) = \mathfrak p_F(u) -\langle u,z\rangle$. Then, by the above
inequality, $e_z(u) \ge 0$ for all $u \in \R^d$. Thus $e_z$ has a finite minimum attained at a unique $u(z) \in \R^d$. Hence $\nabla e_z(u(z))=0$, which is exactly $\nabla \mathfrak p_F(u(z))=z$.
\end{proof}

We will continue to use the notation of this proof: given $\rho \in
\mathrm{int}(\mathscr C_F)$, let $u(\rho)\in \R^d$ be defined by 
$\nabla \mathfrak p_F(u(\rho))=\rho$. Furthermore, we will write $\mu_\rho \in \mathscr M_{\sigma^r}$ for the equilibrium state for $\langle u(\rho),F\rangle$. 

\section{Counting results}\label{sec:counting}

The main results of the paper will follow from counting results for periodic orbits
subject to constraints initiated by Lalley \cite{lalley87} (see also Sharp \cite{Sh92})
and given in streamlined form by  Babillot and Ledrappier \cite{bab-led}.
The results in \cite{bab-led} are stated for hyperbolic flows but the machinery all works for 
suspension flows.

We make the following definitions. Let
$\mathscr A_F$ be the closed subgroup of $\R^d$ generated by
\[
\left\{\int_\eta F \hbox{ : } \eta \in \mathscr P_{\sigma^r}\right\}
\]
and let $\widetilde{\mathscr A_F}$ be the closed subgroup of $\R^{d+1}$ generated by
\[
\left\{\left(\ell(\eta),\int_\eta F\right) \hbox{ : } \eta \in \mathscr P_{\sigma^r}\right\}.
\]
We choose a fundamental domain $\mathfrak F$ for $\mathscr A_F$ and, for $\rho \in \R^d$, 
we define
$\lfloor \rho \rfloor \in \mathscr A_F$ by $\rho -\lfloor \rho \rfloor \in \mathfrak F$.

Let $g_0:\R\to\R$ be a 
 continuous function with compact support
 and let  $g:\mathscr A_F\to\R$ be finitely supported. Following \cite{bab-led}, we will study the functional $N_T^\rho$ defined by  
$$N_T^\rho(g_0\otimes g)=\sum_{\eta\in \mathscr P_{\sigma^r}}
g_0(\ell(\eta)-T)
g\left(\int_\eta F-\lfloor T\rho \rfloor\right).$$
The following is a special case of Theorem 1.2 in \cite{bab-led}.

\begin{proposition}\label{prop:bab_led_result}
Suppose that 
$\widetilde{\mathscr A_F}= \R \times \mathscr A_F$. Then, for all 
$\rho \in \mathrm{int}(\mathscr C_F)$, we have
\[
N_T^\rho(g_0 \otimes g) 
\sim
\frac{\sqrt{|\det \nabla^2\mathfrak h_F(\rho)|}}{(2\pi)^{d/2}}
a_\rho(g_0,g)
\frac{e^{\mathfrak h_F(\rho)T+\langle u(\rho),T\rho-\lfloor T\rho\rfloor \rangle}}{T^{1+d/2}}
,
\]
where
\[
a_\rho(g_0,g)= \int_{\R} e^{\mathfrak p_F(u(\rho))x} g_0(x) \, dx
\sum_{y \in \mathscr A_F}
e^{-\langle u(\rho),y\rangle} g(y) 
,
\]
as $T \to \infty$.
\end{proposition}

We shall apply this result when
$d=b+N$ and $F = \widetilde f$, where $f : \Sigma(\Gamma) \to \Z^{b+N}$ is the
locally constant function defined in
(\ref{hom-function}). Hence, by Proposition \ref{prop:periodic_orbits_generate_homology},
$\mathscr A_F = \Z^d$.
However, there are some additional complications that we need to account for.
First, we note that
\[
\sum_{\gamma\in \mathscr P_{N}}
g_0(\ell(\gamma)-T)
g\left([\gamma]_N-\lfloor T\rho \rfloor\right)
=N_T^\rho(g_0 \otimes g) +O(1),
\]
so the removed orbits do not cause a problem. 
Second, we need to show that $\widetilde{\mathscr A_F}= \R \times \mathscr A_F$. 
We do this in the next lemma.

\begin{lemma}
Suppose that $X^t$ is not locally constant.
Then, for $F$ as in case (i) and case (ii) above, we have
\[
\widetilde{\mathscr A_F}= \R \times \mathscr A_F.
\]
\end{lemma}

\begin{proof}
Clearly, $\widetilde{\mathscr A_F}$ is a closed subgroup of
$\R \times \mathscr A_F = \R \times \Z^d$.
Therefore, it is sufficient to show that if $\chi$ is a character of
$ \R \times \Z^d$ which is trivial on $\widetilde{\mathscr A_F}$ then $\chi$ is trivial
on $\R \times \Z^d$.

The characters of $\R \times \Z^d$ taken the form $\chi_{t,u}(x,y) =
e^{2\pi i (tx+\langle u,y \rangle)}$, with $(t,u) \in \R \times \R^d/\Z^d$.
If $\chi_{t,u}$ is trivial on $\widetilde{\mathscr A_F}$ then
\begin{equation}\label{trivial_character}
\chi_{t,u}\left(\ell(\eta),\int_\eta F\right)
=
\exp 2\pi i\left(t\ell(\eta) + \int_\eta \langle u,F \rangle \right) =1 \quad \forall \eta \in \mathscr P_{\sigma^r}
\end{equation}
Passing to the shift map, this may be rewritten in the form
\begin{equation}\label{trivial_character_shift}
\exp 2\pi i(tr^n(x) + \langle u,f^n(x)\rangle )=1 
\end{equation}
whenever $\sigma^n x=x$ and $n \ge 1$.
By Proposition 5.2 of \cite{PP}, equation (\ref{trivial_character_shift}) may itself be rewritten as
\begin{equation}\label{implies_locally_constant}
2\pi tr +2\pi \langle u,f\rangle + \psi \circ \sigma - \psi = \varphi,
\end{equation}
where $\psi \in C(\Sigma(\Gamma),\R)$ and $\varphi \in C(\Sigma(\Gamma),2\pi \Z)$.
In particular, if (\ref{implies_locally_constant}) holds then $2\pi tr$ is cohomologous to a locally constant function.
Applying the not locally constant condition, this forces $t=0$.

Substituting $t=0$ into (\ref{trivial_character}),we have
 \begin{equation}\label{trivial_character_reduced}
\exp 2\pi i\bigg\langle u,\int_\eta F\bigg\rangle =1 \quad \forall \eta \in \mathscr P_{\sigma^r}.
\end{equation}
Since $\{[\gamma]_N \hbox{ : } \gamma \in \mathscr P\}$ generates the torsion free
part of  $H_1(M_N,\Z)$, we have that
$u =0 \in \mathbb R^d/\mathbb Z^d$.
\end{proof}

Thus we have shown that Proposition \ref{prop:bab_led_result} may be applied in our cases.
We may pass from $N_T^\rho(g_0 \otimes g)$ to a counting function by choosing
$g_0$ to approximate the indicator function of the interval $(-\delta,0]$, say,
and choosing $g$ 
to be the indicator function of an element of $\mathscr A_F$.
We also use the following notation.
Let $C_b(\R^+)$ denote the set set of bounded continuous functions from $\R^+$ to $\R$ and let
\[
C_{b,+}(\R^+) =\left\{f \in C_b(\R^+) \hbox{ : } \inf_{T>0} f(T) >0\right\}.
\]

Here are our main results.

\begin{theorem}\label{thm:main_N}
Let $M$ be a smooth closed oriented Riemannian $3$-manifold and let $X^t : M \to M$ be 
a transitive Anosov flow which is not locally constant. There is a strictly positive
real analytic function $\mathfrak h_N : \mathrm{int}(\mathscr C_N) \to \R^+$ such that, for each
$\rho\in \mathrm{int}(\mathscr C_N)$
and $\alpha \in H_1(M_N,\Z)/\mathrm{torsion}$, we have
\begin{align*}
|\{\gamma \in \mathscr P_N \hbox{ : } T-\delta <\ell(\gamma) \le T, \ [\gamma]_N = \lfloor
T\rho \rfloor +\alpha\}|
\sim c_{\rho,\alpha}(T) \frac{e^{\mathfrak h_N(\rho)T}}{T^{1+(b+N)/2}}, 
\end{align*}
as $T \to \infty$,
where $c_{\rho,\alpha} \in C_{b,+}(\R^+)$.
\end{theorem}

If $0 \in \mathrm{int}(\mathscr C_N)$ then the next corollary gives a result for periodic orbits with fixed 
homology class in $H_1(M_N,\Z)$. However, by work of Dehornoy, there are examples of geodesic flows where this condition fails, since all linking numbers of pairs of periodic orbits are negative
\cite{dehornoy}.

\begin{corollary}
If $0 \in \mathrm{int}(\mathscr C_N)$ then, for every $\alpha \in H_1(M_N,\Z)/\mathrm{torsion}$,
\[
|\{\gamma \in \mathscr P_N \hbox{ : } \ell(\gamma) \le T, \ [\gamma]_N = \alpha\}|
\sim c(m) \frac{e^{\mathfrak h_N(0)T}}{T^{1+(b+N)/2}}, 
\]
as $T \to \infty$, for some $c(m)>0$.
\end{corollary}

\begin{remark}
One may replace the condition $T-\delta < \ell(\gamma) \le T$ with $\ell(\gamma) \le T$ by a simple argument.
We can describe the functions $c_{\rho,\alpha}$ and constants $c(\alpha)$ more explicitly; see \cite{bab-led}, \cite{lalley87},
\cite{lalley}, \cite{Sh92}, \cite{Sh93}.
If $\gamma$ is null-homologous in $M$ then $[\gamma]_N$ allows us to read off the linking numbers
of $\gamma$ with $K_i$, $i=1,\ldots,N$.
\end{remark}

Following the arguments of \cite{bab-led}, 
we can show that the 
periodic orbits considered in Theorem \ref{thm:main_N}
become equidistributed with respect to $\pi_*(\mu_\rho)$,
the projection of $\mu_\rho$ to $M$, as $T \to \infty$ (where $\mu_\rho$ is the equilibrium 
state defined at the end of section \ref{therm_form}).
We have the following result.

\begin{theorem}
Let $M$ be a smooth closed oriented Riemannian $3$-manifold and let $X^t : M \to M$ be 
a transitive Anosov flow which is not locally constant. For each
$\rho\in \mathrm{int}(\mathscr C_N)$
and $\alpha \in H_1(M_N,\Z)/\mathrm{torsion}$, and for every continuous function
$\varphi : M \to \R$, we have
\begin{align*}
\lim_{T \to \infty}
\frac{\displaystyle \sum_{\gamma \in \mathscr P_N \hbox{ : }  T-\delta <\ell(\gamma) \le T, \
 [\gamma]_N = \lfloor
T\rho \rfloor +\alpha}
\frac{1}{\ell(\gamma)} \int_\gamma \varphi}
{|\{\gamma \in \mathscr P_N \hbox{ : } T-\delta <\ell(\gamma) \le T, \ [\gamma]_N = \lfloor
T\rho \rfloor +\alpha\}|}
= \int \varphi \, d\pi_*(\mu_\rho).
\end{align*}
\end{theorem}

We end with a result for amenable covers of $M_N$.
Let $G$ be any quotient group of $\pi_1(M_N)$, with identity element $1_G$. 
Then each periodic orbit 
$\gamma \in \mathscr P_N$ defines a conjugacy class $\langle \gamma \rangle_G$ in $G$.
(The case where $G$ is finite was considered at the start of the paper.) 
Suppose that $G$ sits above $H_1(M_N,\Z)$, i.e. that $G = \pi_1(M_N)/\Lambda$, where
$\Lambda$ is a normal subgroup of the commutator $[\pi_1(M_N),\pi_1(M_N)]$.
Supposing that $0 \in \mathrm{int}(\mathscr C_N)$, the exponential growth rate of periodic orbits for which $\langle \gamma \rangle_G$ trivial is
at most $\mathfrak h_N(0)$.
The results of
\cite{dougall-sharp} may be adapted to show that equality holds if and only if $G$ is amenable.
 More precisely, we have the following theorem.

\begin{theorem}\label{thm:amenable_N}
Let $M$ be a smooth closed oriented Riemannian $3$-manifold and let $X^t : M \to M$ be 
a transitive Anosov flow which is not locally constant and such that 
$0 \in \mathrm{int}(\mathscr C_N)$. 
Let $G = \pi_1(M_N)/\Lambda$, where
$\Lambda$ is a normal subgroup of $[\pi_1(M_N),\pi_1(M_N)]$.
Then
\[
\lim_{T \to \infty} \frac{1}{T} \log
|\{\gamma \in \mathscr P_{N} \hbox{ : } T-\delta <\ell(\gamma) \le T,
\ \langle \gamma \rangle_G 
 = \{1_G\}\}|
\le \mathfrak h_{N}(0).
\]
with equality if and only if $G$ is amenable.
\end{theorem}

\end{document}